\documentclass[12pt]{amsart}

\usepackage{amsmath,amsthm,amsfonts,amssymb,mathrsfs}
\usepackage[all]{xy}

\RequirePackage[dvipsnames,usenames]{color}  

\usepackage[total={6.5in,8.75in}, top=1.0in, bottom=0.5in, left=1.0in, includefoot]{geometry}
\usepackage{amssymb}
\usepackage{latexsym}
\usepackage{amsmath,amsthm}
\usepackage{aeguill}
\usepackage{amssymb}
\usepackage{mathrsfs}
\usepackage{hyperref}
\usepackage{etoolbox}
\usepackage{enumerate}
\usepackage[dvipsnames]{xcolor}
\usepackage{tikz}[2013/12/13]

\makeatletter
\pretocmd{\chapter}{\addtocontents{toc}{\protect\addvspace{15\p@}}}{}{}
\pretocmd{\section}{\addtocontents{toc}{\protect\addvspace{3\p@}}}{}{}
\makeatother

\makeatletter
\def\@tocline#1#2#3#4#5#6#7{\relax
  \ifnum #1>\c@tocdepth 
  \else
    \par \addpenalty\@secpenalty\addvspace{#2}%
    \begingroup \hyphenpenalty\@M
    \@ifempty{#4}{%
      \@tempdima\csname r@tocindent\number#1\endcsname\relax
    }{%
      \@tempdima#4\relax
    }%
    \parindent\z@ \leftskip#3\relax \advance\leftskip\@tempdima\relax
    \rightskip\@pnumwidth plus4em \parfillskip-\@pnumwidth
    #5\leavevmode\hskip-\@tempdima
      \ifcase #1
       \or\or \hskip .5em \or \hskip 1em \else \hskip 1.5em \fi%
      #6\nobreak\relax
    \dotfill\hbox to\@pnumwidth{\@tocpagenum{#7}}\par
    \nobreak
    \endgroup
  \fi}
\makeatother
\DeclareSymbolFont{bbold}{U}{bbold}{m}{n}
\DeclareSymbolFontAlphabet{\mathbbold}{bbold}

\newcommand{\C}{\mathbb{C}}
\newcommand{\N}{\mathbb{N}}
\newcommand{\Z}{\mathbb{Z}}
\newcommand{\R}{\mathbb{R}}
\newcommand{\Q}{\mathbb{Q}}
\newcommand{\F}{\mathbb{F}}

\newcommand{\Gal}{\operatorname{Gal}}

\newcommand{\rk}{\operatorname{rk}}

\renewcommand{\ss}{\operatorname{ss}}

\newcommand{\red}{\operatorname{red}}

\newcommand{\End}{\operatorname{End}}

\newcommand{\GL}{\mathrm{GL}}

\newcommand{\GSp}{\mathrm{GSp}}

\def\sep{\mathrm{sep}}
\def\SS{\mathrm{SS}}

\def\bG{\mathbf{G}}
\def\bH{\mathbf{H}}
\def\bS{\mathbf{S}}
\def\bB{\mathbf{B}}
\def\bT{\mathbf{T}}
\def\bC{\mathbf{C}}

\def\bZ{\mathbf{Z}}
\def\bU{\mathbf{U}}
\def\bY{\mathbf{Y}}

\def\bpx{\begin{pmatrix}}
\def\epx{\end{pmatrix}}

\newcommand{\et}{\mathrm{\text{\'e}t}}

\newtheorem{thm}{Theorem}[section]

\newtheorem{cor}[thm]{Corollary}

\newtheorem{lemma}[thm]{Lemma}

\newtheorem{remark}[thm]{Remark}

\begin{document}

\title[]{On distribution of supersingular primes of abelian varieties and K3 surfaces}

\author{Chun-Yin Hui}
\email{chhui@maths.hku.hk, pslnfq@gmail.com}
\address{
Department of Mathematics\\
The University of Hong Kong\\
Pokfulam, Hong Kong 
}


\thanks{}

\begin{abstract}
Let $X$ be an abelian variety or a K3 surface defined over a number field $K$.
We prove that the density of the supersingular primes of $X$ is zero if $X$ is non-CM. 
By applying an effective Chebotarev density theorem of Serre,
we obtain asymptotic upper bounds of the counting function 
for these supersingular primes.
\end{abstract}

\maketitle

\section{Introduction}
Let $K$ be a number field, $\overline K$ be an algebraic closure of $K$, 
and $\Gal_K:=\Gal(\overline K/K)$ the absolute Galois group. 
Fix a prime number $\ell$.
Given $X$ an abelian variety 
or a K3 surface defined over $K$, let 
\begin{equation*}\label{repn}
\rho_\ell:\Gal_K\to\GL(V_\ell)\simeq \GL_n(\Q_\ell)
\end{equation*}
be the semisimple rational $\ell$-adic Galois representation of $K$,
where $V_\ell:=H^1_{\et}(X_{\overline K},\Q_\ell)$  if $X$ is an abelian variety 
and  $V_\ell:=H^2_{\et}(X_{\overline K},\Q_\ell)$ if $X$ is a K3 surface.
Denote by $\bG_\ell$ the \emph{algebraic monodromy group} of $\rho_\ell$, i.e., 
the Zariski closure 
of the image $\rho_\ell(\Gal_K)$ in $\GL_{n,\Q_\ell}$. 
We say that $X$ is of \emph{CM-type} if the identity component $\bG_\ell^\circ$ is abelian; 
such a condition is independent of $\ell$. 
We say that $X$ is \emph{non-CM} if it is not of CM-type.
By the Tate conjecture for divisors proven in \cite{Fa83},
an abelian variety $X/K$ is of CM-type if and only if
the endomorphism algebra $\End(X_{\overline K})\otimes\Q$ contains 
a $2\dim(X)$-dimensional semisimple commutative 
$\Q$-subalgebra.

Denote by $\Sigma_K$ the set of finite places of $K$
and by $\F_v$ the residue field at $v\in\Sigma_K$.
Let $S_X\subset \Sigma_K$ be the finite subset of places for $X$ to have bad reduction.
If $v\in\Sigma_K\backslash S_X$, we call $v$ a (good) \emph{supersingular} prime
of $X$ if 
the ``reduction mod $v$'' $X_v/\F_v$ (of $X$)  
is a supersingular abelian variety \cite{Oo74} or a supersingular K3 surface \cite{Ar74}.
Denote by $\SS_X$  the set of supersingular primes of $X/K$. 
For any set $Y$, denote its size 
 by $\#Y$. Given a subset $\mathcal P\subset\Sigma_K$, 
define the counting function 
$$\pi_{\mathcal P}(x):=\#\{v\in\mathcal P:~\#\F_v\leq x\}$$ for real $x$. The (natural)
 \emph{density} of $\mathcal P$ is defined as $\lambda\in\R$ if 
\begin{equation}\label{nden}
\pi_{\mathcal P}(x)=\lambda x/\log(x)+o(x/\log(x))\hspace{.1in}\text{as}~x\to\infty.
\end{equation}

For an elliptic curve $E/\Q$,
the density of its supersingular primes is 
zero (resp. $1/2$) if $E$ is non-CM \cite{Se98} (resp. CM \cite{De41}). 
Elkies proves that $E$ has infinitely many supersingular 
primes \cite{El87} (see also \cite{El89}).
When $E$ is non-CM, the Lang-Trotter's conjecture \cite{LT76} 
asserts the existence of a constant $C>0$ (depending on $E$) such that the counting function for supersingular primes satisfies
$$\pi_{\SS_E}(x)\sim C x^{1/2}/\log(x)\hspace{.1in}\text{as}~x\to\infty,$$
see \cite{Se81b},\cite{El91},\cite{FM96} for asymptotic 
upper and lower bounds of $\pi_{\SS_E}(x)$, and see \cite{Ka09} for 
Lang-Trotter's conjecture in the function field case.
For a non-CM abelian surface $A/K$, 
the density of its supersingular primes is zero  by \cite[Theorem 2.3]{Sa16}.
For a principally polarized abelian variety $A/\Q$ 
whose adelic Galois image  is open in $\GSp_{2g}(\hat\Z)$, 
there exists $\gamma>0$ (depending on $g=\dim A$) such that
$$\pi_{\SS_A}(x)=O(x/\log(x)^{1+\gamma-\epsilon})\hspace{.1in}\text{as}~x\to\infty$$
for all $\epsilon>0$,
and under the Generalized Riemann Hypothesis (GRH)\footnote{Here, GRH means any Dedekind zeta function of number field has no zeros $z\in\C$
such that $\mathrm{Re}(z)>1/2$.},
$$\pi_{\SS_A}(x)=O(x^{1-\frac{\gamma}{2}+\epsilon})\hspace{.1in}\text{as}~x\to\infty$$
for all $\epsilon>0$ by Cojocaru-Davis-Silverberg-Stange \cite[Theorem 1(iii)]{CDSS17}\footnote{
The constant $\gamma$ is $1/(2g^2-g+1)$ when $g\neq 2$ and is $1/8$ when $g=2$.
Note that for all sufficiently large prime $p$, 
the trace of $\rho_\ell(Frob_p)$ is zero if $p\in\SS_A$.}.

Motivated by the Lang-Trotter conjecture for elliptic curves, it is interesting to study
the distribution of 
the supersingular primes of an abelian variety 
(resp. a K3 surface) $X/K$, see the recent works \cite{CW23},\cite{Wa24} for certain abelian varieties. On the other hand, one asks if
there exists a finite field extension $L/K$ such that the density of the \emph{ordinary} primes of
the base change $X_{L}$ is one, see positive results \cite{Se98} for elliptic curves, \cite[$\mathsection2$]{Og82} (Katz-Ogus) and \cite{Sa16} for abelian surfaces, \cite{Tan99} and \cite{Pi98} for certain abelian varieties, \cite{BZ09} for K3 surfaces, and a survey \cite{ST25} on recent results concerning reductions of $X$.

In this article, we obtain some general results on the density and distribution of 
the supersingular primes of any abelian variety (resp. K3 surface) $X$ defined over 
any number field $K$ when $X$ is non-CM.

\begin{thm}\label{main1}
Let $X$ be an abelian variety or a K3 surface defined over a number field $K$.
If $X$ is non-CM, then the density of the supersingular primes of $X$ is zero.
\end{thm}

Denote by $\dim\bG_\ell$ (resp. $\rk\bG_\ell$) the dimension (resp. rank) of 
the algebraic monodromy group $\bG_\ell$. It follows from the Mumford-Tate conjecture  
that $\dim\bG_\ell$ is independent of $\ell$. This conjecture is open for abelian varieties and
is known for K3 surfaces \cite{Tan91,Tan95}.
It is also known that the connectedness and the rank of $\bG_\ell$ in Theorem \ref{main2}(ii) 
are independent of $\ell$ \cite{Se81a}.

\begin{thm}\label{main2}
Let $X$ be an abelian variety or a K3 surface defined over a number field $K$.
If $X$ is non-CM, then the following assertions hold for all $\epsilon>0$.
\begin{enumerate}[(i)]
\item  We have
$$\pi_{\SS_X}(x)=O(x/\log(x)^{1+ \frac{1}{\dim\bG_\ell}-\epsilon})\hspace{.1in}\text{as}~x\to\infty,$$
and under GRH, we have
$$\pi_{\SS_X}(x)=O(x^{1-\frac{1}{2\dim\bG_\ell}+\epsilon})\hspace{.1in}\text{as}~x\to\infty.$$
\item Suppose $\bG_\ell$ is connected. We have 
$$\pi_{\SS_X}(x)=O(x/\log(x)^{1+ \frac{\rk\bG_\ell-1}{\dim\bG_\ell}-\epsilon})\hspace{.1in}\text{as}~x\to\infty,$$
and under GRH, we have
$$\pi_{\SS_X}(x)=O(x^{1-\frac{\rk\bG_\ell-1}{2\dim\bG_\ell}+\epsilon})\hspace{.1in}\text{as}~x\to\infty.$$
\item The sum below  is convergent:
$$\sum_{v\in\SS_X} \frac{1}{\#\F_v}.$$
\end{enumerate}
\end{thm}

The main tools for Theorems \ref{main1} and \ref{main2} are respectively, the
$\ell$-adic Galois representations of 
smooth projective varieties over global fields (see $\mathsection2$)
and an effective Chebotarev density theorem for $\ell$-adic Galois extensions \cite{Se81b} 
(see $\mathsection3$).
The method for Theorem \ref{main1} 
can be adapted to general smooth projective varieties $X$ defined over global fields 
by Theorem \ref{prop} and Corollary \ref{cor}, 
once  a suitable definition of supersingular reduction  of $X$ is available (see e.g., \cite{FL21} for hyperk\"ahler varieties).

\section{$\ell$-adic Galois representations}
\subsection{} 
Let $K$ be a global field, $K^{\sep}$ be a separable closure of $K$, and $\Gal_K:=\Gal(K^{\sep}/K)$
the absolute Galois group of $K$.
Let $X$ be a smooth projective variety defined over $K$
and $w\geq 0$ be an integer.
By considering the continuous action of $\Gal_K$ on 
the $\ell$-adic cohomology groups $V_\ell:=H^w_{\et}(X_{K^{\sep}},\Q_\ell)$
for all $\ell$ not equal to the characteristic of $K$,
we obtain a \emph{compatible system} 
\begin{equation}\label{cs}
\{\rho_\ell:\Gal_K\to\GL(V_\ell)\simeq \GL_{n}(\Q_\ell)\}_{\ell\neq \mathrm{Char}(K)}
\end{equation}
of $\ell$-adic representations by \cite{De74},
i.e., there exist a finite (exceptional set) $S_X\subset \Sigma_K$
and polynomials $\Phi_v(T)\in\Q[T]$ for each $v\in\Sigma_K\backslash S_X$ such that 
the following conditions hold.
\begin{enumerate}[(a)]
\item For each $\ell$, the representation $\rho_\ell$ is unramified outside $S_X$ and $S_\ell:=\{v\in\Sigma_K:~ v~\text{divides}~\ell\}$.
\item For each $\ell$ and $v\in\Sigma_K\backslash (S_X\cup S_\ell)$,
the characteristic polynomial $\det(\rho_\ell(Frob_v)-TI_n)$
is equal to $\Phi_v(T)$, where $\rho_\ell(Frob_v)$
denotes the image of the (geometric) Frobenius conjugacy class at $v$.
\end{enumerate}

If $\beta$ is a root of $\Phi_v(T)\in\Q[T]$ and $p_v$ denotes the residue characteristic of $v$,
then the compactness of $\Gal_K$ and condition (b) above imply that
$\beta$ is an $\ell$-adic unit for all prime $\ell$ not equal
to $p_v$.
Moreover, $\beta$ is a $\#\F_v$-Weil number of weight $w$ 
by the Weil conjecture proven by Deligne \cite{De74}, i.e., 
the complex absolute value $|i(\beta)|=(\#\F_v)^{w/2}$ for any embedding $i:\overline\Q\to \C$.
The representation $\rho_\ell$ is conjectured to be semisimple by Grothendieck-Serre \cite{Tat65}.
This conjecture is known when $X$ is an abelian variety by \cite{Tat66},\cite{Za75},\cite{Fa83} 
or a K3 surface by \cite{PSS75},\cite{De72},\cite{De80}.

The image of $\rho_\ell$ is a compact $\ell$-adic Lie subgroup of $\GL_n(\Q_\ell)$.
Let $\bG_\ell\subset\GL_{n,\Q_\ell}$ be the algebraic monodromy group of $\rho_\ell$,
$\bU_\ell$ be the unipotent radical of $\bG_\ell$, and 
$\bG_\ell^{\red}$ the reductive quotient $\bG_\ell/\bU_\ell$.
When $K$ is a number field, the Mumford-Tate conjecture \cite{Mu66},\cite{Se94}
implies that the identity component $\bG_\ell^\circ$ is independent of $\ell$.
It is known that the rank and semisimple rank of $\bG_\ell^{\red}$ are independent of $\ell$ \cite{Se81a},\cite{Hu13}. Moreover, if $(\bG_\ell^{\red})^\circ$ is abelian for one $\ell$, 
then this is true for all $\ell$ \cite[Chapter III]{Se98}.
When $K$ is a positive characteristic global field, the $\ell$-independence of $(\bG_\ell^{\red})^\circ$
is obtained in \cite{Ch04}.

\subsection{}
We first prove a lemma on the following $r$-\emph{unipotent subvariety} $\bY$ 
of some reductive group $\bH$, where $r\in\N$.

\begin{lemma}\label{lem}
Let  $\bH$ be a reductive group defined over a field $k$  
and $r$ be a positive integer.
If the identity component $\bH^\circ$ is non-abelian, then
the following assertions hold.
\begin{enumerate}[(i)]
\item The closed subvariety
$\bY:=\{h\in \bH: ~h^r~\text{is unipotent}\}$
is nowhere dense in $\bH$.
\item Let $\bY_1:=\{h\in \bH: ~h^r=1\}\subset \bY$ be
the $r$-torsion (closed) subvariety. 
If $\mathrm{Char}(k)=0$ and $\bH$ is connected, then 
$$\dim\bY_1\leq \dim\bH-\rk\bH.$$
\end{enumerate}
\end{lemma}

\begin{proof}
To prove the lemma, we may assume that $k$ is algebraically closed. 
If (i) is false, we obtain
$h\bH^\circ\subset \bY$ for some element $h$. There exist $h'\in h\bH^\circ$,
a maximal torus $\bT$ of $\bH$, and a Borel subgroup $\bB$ of $\bH$ containing $\bT$
such that the pair $(\bB,\bT)$ is invariant under conjugation by $h'$ (as an automorphism of $\bH^\circ$).
It follows that the sum of the positive co-roots 
(non-empty as the reductive group $\bH^\circ$ is not a torus) 
corresponding to the root datum of the triple $(\bH^\circ,\bB,\bT)$ is fixed by $h'$ under conjugation.
Then $h'$ commutes with some one-dimensional subtorus $\bT'$ of $\bT$ and we can
pick $t\in\bT'$ such that $t^r\neq 1$. Since $h't\in h\bH^\circ\subset \bY$,
it follows from the definition of $\bY$ that
$$u:=(h't)^r= (h')^r t^r$$
is unipotent. As the unipotent $(h')^r$ commutes with the semisimple $t^r$,
the uniqueness of the Jordan decomposition (\cite[Chapter 1, $\mathsection4$]{Bor91}) implies that $(h')^r=u$ and $t^r=1$,
which is a contradiction.

Let $\bT$ be a maximal torus of $\bH$.
When $\mathrm{Char}(k)=0$ and $\bH$ is connected, every element of $\bY_1$ is conjugate to an element 
of $\bT\cap\bY_1$ which is finite. Since the $\bH$-orbit of an element of $\bT$ (under conjugation)
has dimension bounded above by $\dim\bH-\dim\bT=\dim\bH-\rk\bH$,
we obtain (ii).
\end{proof}

\subsection{}\label{2.3}
Fix a representation $\rho_\ell$ in the compatible system \eqref{cs}.
We may take the exceptional set $S_X$ to be the places for $X$ to have bad reduction.  
 If $v\in\Sigma_K\backslash (S_X\cup S_\ell)$, then $\rho_\ell$ is unramified at $v$ 
and the conjugacy class
\begin{equation}\label{Frob}
\rho_\ell(Frob_v)\subset \bG_\ell(\Q_\ell)\subset\GL_n(\Q_\ell).
\end{equation} 
For $x\in \bG_\ell(\Q_\ell)$, 
denote by $x^{\ss}\in\bG_\ell(\Q_\ell)$ its semisimple part.
Let $f_v\in \rho_\ell(Frob_v)$.
Define the following $\Q_\ell$-subgroups of $\GL_{n,\Q_\ell}$.
\begin{itemize}
\item $\bS_{f_v}$: the smallest algebraic subgroup containing $f_v$.
\item $\bS_{f_v}^{\red}$: the smallest algebraic subgroup containing $f_v^{\ss}$.
\item $\bT_{f_v}$: the identity component of $\bS_{f_v}^{\red}$.
\item $\mathbb{G}_m$: the subgroup of homotheties in $\GL_{n,\Q_\ell}$.
\end{itemize}
The diagonalizable group $\bS_{f_v}^{\red}$ is the reductive quotient of $\bS_{f_v}$. 
The torus $\bT_{f_v}$ is called the \emph{Frobenius torus} at $f_v$ \cite{Se81a}.
Define the following conjugation-invariant closed $\Q_\ell$-subvarieties of $\bG_\ell$.
\begin{itemize}
\item $\bC$: the Zariski closure in $\bG_\ell$ of the subset
\begin{equation}\label{ssset}
\{ x\in \rho_\ell(\Gal_K):~x\in \rho_\ell(Frob_v),~~v\in\Sigma_K\backslash (S_X\cup S_\ell),~~
\text{and}~~\bT_{f_v}=\mathbb{G}_m\}.
\end{equation}
\item $\bC_1$: the Zariski closure in $\bG_\ell$ of the subset
\begin{equation}\label{ssset1}
\{ x\in \rho_\ell(\Gal_K):~x\in \rho_\ell(Frob_v),~~ x=x^{\ss},~~v\in\Sigma_K\backslash (S_X\cup S_\ell),~~
\text{and}~~\bT_{f_v}=\mathbb{G}_m\}.
\end{equation}
\end{itemize}
Note that \eqref{ssset} and \eqref{ssset1} are conjugation-invariant
subsets of the Galois image $\rho_\ell(\Gal_K)$ and $\bC_1$ 
is a subvariety of $\bC$.

\begin{thm}\label{prop}
Let $\rho_\ell$ be the $\ell$-adic representation of the global field $K$ 
attached to $H^w_{\et}(X_{K^{\sep}},\Q_\ell)$.
 If the identity component of 
$\bG_\ell^{\red}$ is non-abelian, then the following assertions hold.
\begin{enumerate}[(i)]
\item The above subvariety $\bC$ is nowhere dense in $\bG_\ell$.
\item If $\bG_\ell$ is connected, then $\dim\bC_1\leq \dim\bG_\ell-\rk\bG_\ell+1$.
\end{enumerate}
\end{thm}

\begin{proof}
It is known that $\mathbb{G}_m\subset \bG_\ell$ (noted by Deligne, see \cite{Se81a}).
Since the identity component of 
$\bG_\ell^{\red}=\bG_\ell/\bU_\ell$ is non-abelian, this is also true for the quotient
$$\bH:=\bG_\ell/(\bU_\ell \mathbb{G}_m).$$ 

 By \cite[Lemma 1.3(c)]{LP97} (due to Serre), there are 
finitely many possibilities (up to conjugation) for
the diagonalizable subgroup
$$\bS_{f_v}^{\red}\times\overline\Q_\ell\subset\GL_{n,\overline\Q_\ell}$$
when $v$ runs through all the primes in $\Sigma_K\backslash (S_X\cup S_\ell)$ 
and $f_v\in \rho_\ell(Frob_v)$.
Thus, there exists $r\in\N$ such that 
$$(f_v^{\ss})^r\in \bT_{f_v}=\bS_{f_v}^\circ$$
for all $v\in\Sigma_K\backslash (S_X\cup S_\ell)$ and $f_v\in \rho_\ell(Frob_v)$.
Hence, Jordan decomposition implies that 
the image of \eqref{ssset} in $\bH$ via the surjection
$$\pi:\bG_\ell\to\bG_\ell^{\red}\to \bH$$
 is contained in the subvariety $\bY$ in Lemma \ref{lem}(i),
which is nowhere dense in $\bH$. We obtain (i) since $\bC$ is contained in 
the preimage $\pi^{-1}(\bY)$, which is also nowhere dense in $\bG_\ell$.

If $\bG_\ell$ is connected, then Lemma \ref{lem}(ii) implies that
\begin{align}\label{long}
\begin{split}
\dim\bY_1\leq \dim\bH - \rk\bH&=(\dim\bG_\ell-\dim \bU_\ell - 1)-(\rk\bG_\ell-1)\\
&=\dim\bG_\ell-\dim\bU_\ell-\rk\bG_\ell.
\end{split}
\end{align}
Since $\bC_1\subset\pi^{-1}(\bY_1)$,
we obtain 
\begin{equation}\label{short}
\dim\bC_1\leq \dim\bY_1+\dim\bU_\ell+1.
\end{equation}
Hence, \eqref{long} and \eqref{short} imply (ii).
\end{proof}

\subsection{} Let $K$ be a global field and $\mathcal P\subset\Sigma_K$ be a subset.
In this section, the density of $\mathcal P$ means 
the natural density (see \eqref{nden}) if $\mathrm{Char}(K)=0$, and 
it means the Dirichlet density if $\mathrm{Char}(K)>0$, i.e., the density is $\lambda\in\R$ if 
$$\lim_{s\to 1^+}\frac{\sum_{v\in\mathcal P}\frac{1}{(\#\F_v)^s}}{\sum_{v\in\Sigma_K}\frac{1}{(\#\F_v)^s}}=\lambda.$$

We state the Chebotarev density theorem for global fields.

\begin{thm}\label{CDT}
Let $K$ be a global field and $L$ be a Galois field extension of $K$
that is unramified outside a finite $S\subset\Sigma_K$. Let $C$ be a conjugation 
invariant subset of $\Gal(L/K)$ (equipped with normalized Haar measure $\mu$)
such that the boundary of $C$ is of measure zero. Then the density of 
the subset
$$\{v\in\Sigma_K\backslash S: ~Frob_v\subset C\}\subset\Sigma_K$$
is equal to the measure $\mu(C)$.
\end{thm}

\begin{remark}
The proof when $K$ is a number field (using natural density) 
can be found in \cite[$\mathsection6.2.1$]{Se12}.
The key point is to show that the statement holds when $L/K$ is finite Galois.
Since this is true for global fields by using Dirichlet density,
the proof (in \cite[$\mathsection6.2.1$]{Se12}) 
can be carried to global fields in positive characteristic.
\end{remark}

\begin{cor}\label{cor}
Let $\rho_\ell$ be the $\ell$-adic representation of the global field $K$ 
attached to $H^w_{\et}(X_{K^{\sep}},\Q_\ell)$.
Define the conjugation invariant closed-analytic 
subset
\begin{equation}\label{Cset}
C:=\bC\cap \rho_\ell(\Gal_K)
\end{equation}
of $\rho_\ell(\Gal_K)$, where $\bC\subset\bG_\ell$ is the subvariety in $\mathsection\ref{2.3}$. 
If the identity component of 
$\bG_\ell^{\red}$ is non-abelian,
then the subset
\begin{equation}\label{Cprime}
\{v\in\Sigma_K\backslash (S_X\cup S_\ell):~\rho_\ell(Frob_v)\subset C\}\subset\Sigma_K
\end{equation}
is of density zero.
\end{cor}

\begin{proof}
We have $\rho_\ell(\Gal_K)=\Gal(L/K)$ for some Galois extension $L/K$
that is unramified outside the finite subset $S:=S_X\cup S_\ell\subset\Sigma_K$. 
Equip the compact $\ell$-adic Lie group $\rho_\ell(\Gal_K)$ with the normalized Haar measure $\mu$.
By the Chebotarev density theorem (Theorem \ref{CDT}),
it suffices to show that 
the conjugation invariant closed-analytic 
subset $C$ has measure zero.
Since the identity component of $\bG_\ell^{\red}$ is non-abelian, 
$C$ contains no interior point of $\rho_\ell(\Gal_K)$ by Theorem \ref{prop}(i). Therefore, 
we have $\mu(C)=0$ by \cite[Corollary 5.10]{Se12}.
\end{proof}

\subsection{}
We are now ready to prove Theorem \ref{main1}.
The smooth projective variety $X$ is a non-CM abelian variety 
or non-CM K3 surface defined over a number field $K$.
Recall that $\SS_X$ denotes the set of supersingular primes of $X/K$.

\vspace{.1in}

\noindent\textbf{Proof of Theorem \ref{main1}.}
When $X/K$ is an abelian variety (resp. a K3 surface),
 we take $w=1$ (resp. $w=2$) in Theorem \ref{prop}.
Since the $\ell$-adic representation $\rho_\ell$ is semisimple, the algebraic 
monodromy group $\bG_\ell$ is reductive. Moreover, any $f_v$ in \eqref{Frob}
is semisimple, i.e., $f_v=f_v^{\ss}$.
Since $X$ is non-CM, $\bG_\ell^\circ$ is non-abelian.
By Corollary \ref{cor}, it suffices
to show that $\SS_X\backslash S_\ell$ is a subset of \eqref{Cprime}.

Let $v\in\SS_X\backslash S_\ell$ 
be a supersingular prime and $\beta$ be a root
of the characteristic polynomial $\Phi_v(T)$. 
Then the fraction 
\begin{equation}\label{fraction}
\frac{\beta}{(\#\F_v)^{w/2}}\in\overline\Q
\end{equation} 
is a $p_v$-adic unit, where $p_v$ is the residue characteristic of $v$.
Since \eqref{fraction}
is also a Weil number of weight zero
and is an $\ell$-adic unit for all prime $\ell\neq p_v$, 
Kronecker's theorem implies that \eqref{fraction} 
is a root of unity. Hence, the Frobenius torus $\bT_{f_v}$
is $\mathbb{G}_m$ for all $v\in\SS_X\backslash S_\ell$.
We are done.\qed

\section{Effective Chebotarev density theorem}
Let $K$ be a number field and $L$ be a Galois extension of $K$ that is unramified 
outside a finite $S\subset\Sigma_K$.
Suppose $G:=\Gal(L/K)$ is a compact $\ell$-adic Lie group of dimension $N\geq 1$
and $C$ is a conjugation-invariant closed subset of $G$.
Define the counting function
$$\pi_C(x):=\#\{v\in\Sigma_K\backslash S:~Frob_v\subset C~\text{and}~\#\F_v\leq x\}$$
for real $x$.
We need the following effective Chebotarev density theorem 
for number fields\footnote{This relies on effective
versions of the Chebotarev density theorem \cite{LO77},\cite{LMO79}}
to prove Theorem \ref{main2}.

\begin{thm}\label{eCDT}\cite[Corollary 1 of Theorem 10]{Se81b}
Let $d$ be a real number such that $0\leq d<N$.
Suppose $\dim_M C\leq d$ (see the definition in \cite[$\mathsection3.1$]{Se81b}) and write $\alpha=(N-d)/N$. 
\begin{enumerate}[(i)]
\item For all $\epsilon>0$, we have
$$\pi_{C}(x)=O(x/\log(x)^{1+ \alpha-\epsilon})\hspace{.1in}\text{as}~x\to\infty.$$
\item Assume GRH. For all $\epsilon>0$, we have
$$\pi_{C}(x)=O(x^{1-\frac{\alpha}{2}+\epsilon})\hspace{.1in}\text{as}~x\to\infty.$$
\end{enumerate}
\end{thm}

\noindent \textbf{Proof of Theorem \ref{main2}.}
Now $K$ is a number field. 
To apply Theorem \ref{eCDT}, we take 
\begin{itemize}
\item $G:=\rho_\ell(\Gal_K)=\Gal(L/K)$ 
(as in the proof of Corollary \ref{cor}),
\item $S:=S_X\cup S_\ell$,
\item $C:=\bC\cap \rho_\ell(\Gal_K)$ in \eqref{Cset},
\item $N:=\dim G$.
\end{itemize}
Recall that $\bG_\ell$ is reductive and $\bG_\ell^\circ$ is non-abelian as $X$ is non-CM.
Since  $f_v=f_v^{\ss}\in\rho_\ell(Frob_v)$
for $v\in\Sigma_K\backslash S$, 
the conjugation invariant $\Q_\ell$-subvarieties $\bC$ and $\bC_1$ in $\mathsection\ref{2.3}$
coincide. 
 
Since $G$ is open in $\bG_\ell(\Q_\ell)$ by \cite{Bog80},
we obtain $N=\dim\bG_\ell$. By \cite[Theorem 8]{Se81b},
we have $\dim_M C\leq \dim C$ (the dimension as an $\ell$-adic variety). 
Since we also have $\dim C\leq\dim\bC(\Q_\ell)\leq \dim\bC$\footnote{
The second inequality can be seen by first decomposing $\bC$ as a finite 
disjoint union of locally closed smooth
$\Q_\ell$-subvarieties $\bZ_i\subset\bG_\ell$ 
and then applying implicit function theorem at points of $\bZ_i(\Q_\ell)$.},
we obtain
$$\dim_M C\leq \dim\bC.$$

For assertion (i), we take $d:=\dim\bG_\ell-1\geq \dim\bC$ by Theorem \ref{prop}(i).
For assertion (ii), we take $d:=\dim\bG_\ell-\rk\bG_\ell+1\geq \dim\bC_1=\dim\bC$ by Theorem \ref{prop}(ii)
as $\bG_\ell$ is connected.
It follows that $0\leq d<N$ and $\dim_M C\leq d$. 
Since we have
$$\pi_{\SS_X}(x)\leq \#S_\ell+\pi_{\SS_X\backslash S_\ell}(x)\leq \#S_\ell+
 \pi_C(x)$$ for all $x$ (as $\SS_X\backslash S_\ell\subset \eqref{Cprime}$),
Theorem \ref{eCDT} implies assertions (i) and (ii).

Finally, assertion (iii) follows from assertion (i) and 
\cite[Corollary 2 of Theorem 10]{Se81b}. \qed

\vspace{.2in}

\section*{Acknowledgments}
I would like to thank Michael Larsen for inspiration.
I am grateful to Jean-Pierre Serre 
for suggesting to use the Dirichlet density for places of 
global fields in positive characteristic and a remark on applying \cite[Theorem 10]{Se81b}.
I am grateful to Yuri Zarhin for introducing \cite{BZ09} and some comments.
I thank Anna Cadoret, Kazuhiro Ito, Yuk-Kam Lau, Max Lieblich, Ben Moonen, and 
Matthias Sch\"utt for their comments.

Hui was partially supported by Hong Kong RGC (no. 17314522), NSFC (no. 12222120), and a Humboldt Research Fellowship.

\vspace{.1in}

\begin{thebibliography}{BLGGT14}
\bibitem[Ar74]{Ar74}
Michael Artin:
Supersingular K3 surfaces,\textit{ Ann. Sci. \'Ecole Norm. Sup.} (4) 7, 543--567 (1974).

\bibitem[Bog80]{Bog80}
Fedor A. Bogomolov: Sur l'alg\'ebricit\'e des repr\'esentations $\ell$-adiques, \textit{C.R.A.S.} \textbf{290} (1980), 701-703.

\bibitem[Bor91]{Bor91}
Armand Borel:
Linear Algebraic Groups, Graduate Texts in Mathematics, 126 (2nd ed.), Springer-Verlag 1991.

\bibitem[BZ09]{BZ09}
Fedor A. Bogomolov and Yuri G. Zarhin
Ordinary reduction of K3 surfaces,
\textit{Cent. Eur. J. Math.}, 7 (2) (2009), pp. 206-213.

\bibitem[CDSS17]{CDSS17}
Alina Carmen Cojocaru, Rachel Davis, Alice Silverberg, and Katherine E. Stange:
 Arithmetic properties
of the Frobenius traces defined by a rational abelian variety (with two appendices by J.-P. Serre), 
\textit{IMRN}, (12):3557--3602, 2017.

\bibitem[CW23]{CW23}
Alina Carmen Cojocaru and  Tian Wang: Bounds for the distribution of the Frobenius traces associated to a generic abelian variety, preprint, 2023. 


\bibitem[Ch04]{Ch04}
Chee Whye Chin:
	Independence of $\ell$ of monodromy groups,
	\textit{J.A.M.S.}, Volume \textbf{17}, Number 3, 723-747.
	
\bibitem[De41]{De41}
Max Deuring:
Die Typen der Multiplikatorenringe elliptischer Funktionenk\"orper, Abh. Math. Sem. Hansischen Univ. \textbf{14},
1941, 197--272.

\bibitem[De72]{De72} 
  Pierre Deligne: 
La conjecture de Weil pour les surfaces K3,
\textit{Inventiones Math.} \textbf{15}, 206--226 (1972).

\bibitem[De74]{De74} 
  Pierre Deligne: 
  La conjecture de Weil I, 
  \textit{Publ. Math. I.H.E.S.}, \textbf{43} (1974), 273-307.
	
\bibitem[De80]{De80} 
  Pierre Deligne: 
  La conjecture de Weil II, 
  \textit{Publ. Math. I.H.E.S.}, \textbf{52} (1980), 137-252.	

\bibitem[El87]{El87}
Noam D. Elkies:
The existence of infinitely many supersingular primes
for every elliptic curve over $\Q$, \textit{Invent. Math.} \textbf{89} (1987), no. 3, 561--567.


\bibitem[El89]{El89}
Noam D. Elkies:
 Supersingular primes for elliptic curves over real
number fields, \textit{Compositio Math.} \textbf{72} (1989), no. 2, 165--172.

\bibitem[El91]{El91}
Noam D. Elkies:
Distribution of supersingular primes, Journ\'ees Arithm\'etiques (1989): Luminy, Ast\'erisque
No. 198-200, 1991, pp. 127--132.

\bibitem[Fa83]{Fa83} 
Gerd Faltings: 
Endlichkeitss$\ddot{\mathrm{a}}$tze f$\ddot{\mathrm{u}}$r abelsche Variet$\ddot{\mathrm{a}}$ten $\ddot{\mathrm{u}}$ber Zahlk$\ddot{\mathrm{o}}$rpern,  \textit{Invent. Math.} \textbf{73} (1983), no. 3, 349-366. 

\bibitem[FM96]{FM96}
\'E. Fouvry and M. R. Murty:
 On the distribution of supersingular primes, \textit{ Canadian Journal of Mathematics} \textbf{48}, no 1,
1996, 81--104.

\bibitem[FL21]{FL21}
Lie Fu and Zhiyuan Li:
Supersingular irreducible symplectic varieties,  
``Rationality of Algebraic Varieties'', Progress in Mathematics, Vol. 342, 191--244, (2021).

\bibitem[Hu13]{Hu13}
	Chun Yin Hui:
	Monodromy of Galois representations and equal-rank subalgebra equivalence,
	\textit{Math. Res. Lett.} \textbf{20} (2013), no. 4, 705--728.
	
\bibitem[Ka09]{Ka09}
Nicholas M. Katz: Lang-Trotter revisited, \textit{Bulletin of the American Mathematical Society} 
\textbf{46}, No. 3, 2009, pp. 413--457.

\bibitem[LMO79]{LMO79}
J. C. Lagarias, H. L. Montgomery, and A. M. Odlyzko:
 A bound for the least prime ideal in the Chebotarev
density theorem, \textit{Invent. Math.}, \textbf{54}(3):271--296, 1979.


\bibitem[LO77]{LO77}
J. C. Lagarias and A. M. Odlyzko:
 Effective versions of the Chebotarev density
theorem, in Algebraic Number Fields, edit. A. Fr\"ohlich, Academic Press, London, 1977,
409--464.

\bibitem[LT76]{LT76}
Serge Lang and Hale Trotter:
 Frobenius distributions in $\GL_2$-extensions,
Springer Lecture Notes in Mathematics 504, 1976.

\bibitem[LP97]{LP97} 
Michael	Larsen and Richard Pink: 
 	A connectedness criterion for $\ell$-adic Galois representations,
	\textit{Israel J. Math.} \textbf{97} (1997), 1--10.

\bibitem[Mu66]{Mu66} 
  David Mumford: 
	Families of abelian varieties, \textit{Algebraic Groups and Discontinuous
Subgroups (Boulder, CO, 1965)}, pp. 347-351, Proc. Sympos. Pure Math., Vol.
9, \textit{Amer. Math. Soc., Providence, RI}, 1966.

\bibitem[No97]{No97}
Rutger Noot:
Abelian varieties - Galois representation and properties of ordinary reduction, 
\textit{Compositio Math.}, 1995, \textbf{97}, 161--171.

\bibitem[No00]{No00}
Rutger Noot:
 Abelian varieties with $\ell$-adic Galois representation of Mumford's type, 
\textit{J. Reine Angew. Math.}, 2000, \textbf{519},
155-169.


\bibitem[Og82]{Og82}
Arthur Ogus:
Hodge Cycles and Crystalline Cohomology, Hodge Cycles, Motives, and Shimura Varieties, Lecture Notes in Mathematics, vol. 900, Springer-Verlag, Berlin-New York (1982), pp. 357--414.

\bibitem[Oo74]{Oo74}
Frans Oort:
Subvarieties of moduli spaces, \textit{Invent. Math.} \textbf{24} (1974), 95--119.

\bibitem[PSS75]{PSS75}
I. I. Piatetski-Shapiro and I. R. Shafarevich: 
The Arithmetic of K3 surfaces, \textit{Proc. Steklov
Inst. Math.} \textbf{132} (1975), 45--57.

\bibitem[Pi98]{Pi98}
Richard Pink:
$l$-adic algebraic monodromy groups, cocharacters, and the Mumford-Tate conjecture,
\textit{J. Reine Angew. Math.}, 495 (1998), pp. 187-237.


\bibitem[Sa16]{Sa16}
William F. Sawin:
Ordinary primes for Abelian surfaces
\textit{C. R. Math. Acad. Sci. Paris}, 354 (6) (2016), pp. 566--568.

\bibitem[Se81a]{Se81a}
Jean-Pierre Serre:
Letter to K. A. Ribet, Jan. 1, 1981, Oeuvres-Collected Papers IV, no. 133.

\bibitem[Se81b]{Se81b}
Jean-Pierre Serre:
  Quelques applications du th\'eor\`eme de densit\'ede Chebotarev, \textit{Inst. Hautes \'Etudes
Sci. Publ. Math.}, \textbf{54} (1981), 323--401.

\bibitem[Se94]{Se94}
Jean-Pierre Serre:
	Propri\'et\'es conjecturales des groupes de Galois motiviques et des repr\'esentations $l$-adiques. Motives (Seattle, WA, 1991), 377--400, Proc. Sympos. Pure Math., 55, Part 1, Amer. Math. Soc., Providence, RI, 1994. 
	
\bibitem[Se98]{Se98} 
Jean-Pierre Serre:
	Abelian $l$-adic representation and elliptic curves, Research Notes in Mathematics Vol. 7 (2nd ed.), \textit{A K Peters} (1998).
	
\bibitem[Se12]{Se12} 
Jean-Pierre Serre:
	Lectures on $N_X(p)$, Chapman \& Hall/CRC Research Notes in Mathematics 11,
CRC Press, Boca Raton, FL, 2012.



\bibitem[ST25]{ST25}
Ananth N. Shankar and Yunqing Tang: Reductions of abelian varieties and K3 surfaces, \textit{Journal of Number Theory}, \textbf{270} (2025) 122--166.

\bibitem[Tan91]{Tan91}
S.G. Tankeev: 
Surfaces of K3 type over number fields and the Mumford-Tate conjecture I, 
Math. USSR Izv. \textbf{37} (1991), no. 1, 191-208. 

\bibitem[Tan95]{Tan95}
S.G. Tankeev:
Surfaces of K3 type over number fields and the Mumford-Tate conjecture II,
Math. USSR-Izv.  \textbf{59} (1995), no. 3, 619-646.

\bibitem[Tan99]{Tan99}
S.G. Tankeev:
On the weights of the $\ell$-adic representation and arithmetic of Frobenius eigenvalues, 
\textit{Izv. Ross. Akad. Nauk Ser. Mat.}, 1999, 63, 185-224 (in Russian), Izv. Math., 1999, 63, 181-218.


\bibitem[Tat65]{Tat65}
John Tate:
Algebraic cohomology classes and poles of zeta functions, in Schilling, O. F. G., Arithmetical Algebraic
Geometry (Proc. Conf. Purdue Univ., 1963), New York: Harper and Row, 1965, p. 93-110.

\bibitem[Tat66]{Tat66}
John Tate:
Endomorphisms of abelian varieties over finite fields, \textit{Invent. Math.}, 
\textbf{2} (1966), 134--144.



\bibitem[Wa24]{Wa24} Tian Wang: Distribution of supersingular primes for abelian surfaces, preprint, 2024.

\bibitem[Za75]{Za75}
Yu. G. Zarhin:
 Endomorphisms of abelian varieties over fields of finite characteristic, 
\textit{Izv. Akad. Nauk SSSR Ser. Math.} \textbf{39}, 1975, no.2, 272--277, 471.
	
	\end{thebibliography}
\end{document}